\documentclass[14pt]{extarticle}
\usepackage{extsizes}
\usepackage{amssymb,amsfonts,amsthm}
\usepackage{amsmath}
\usepackage[makeroom]{cancel}
\usepackage{mathtools}
\usepackage{mathrsfs,pifont}
\usepackage{slashed,mathabx} 
\usepackage[bbgreekl]{mathbbol}
\usepackage{tikz}
\usepackage{framed}
\usetikzlibrary{calc}

\usepackage[top=1in, bottom=1.25in, left=0.75in,right=0.75in]{geometry}
\usepackage{authblk}

\usepackage[font={small,sl}]{caption}
\usepackage[all]{xy}

\usepackage{hyperref}

\newtheorem{theorem}{Theorem}[section]
\newtheorem{lemma}[theorem]{Lemma}

\theoremstyle{definition}

\newtheorem{conjecture}{Conjecture}

\theoremstyle{remark}

\numberwithin{equation}{section}

\makeatletter
\def\blfootnote{\xdef\@thefnmark{}\@footnotetext}
\makeatother

\begin{document}

\title{A note on some spectral properties of \\ generalised pancake graphs}

\author{Gary Greaves\thanks{Division of Mathematical Sciences, School of Physical and Mathematical Sciences, Nanyang Technological University, 21 Nanyang Link, Singapore 637371. Email: gary@ntu.edu.sg.} \qquad Haoran Zhu\thanks{Division of Mathematical Sciences, School of Physical and Mathematical Sciences, Nanyang Technological University, 21 Nanyang Link, Singapore 637371. Email: zhuh0031@e.ntu.edu.sg} }

\date{}

\maketitle

\begin{abstract}
We prove that the spectral gap of generalised pancake graphs is strictly less than 2 and strictly less than 1 for burnt pancake graphs. 
In addition, we establish lower bounds on the multiplicities of certain integer eigenvalues of generalised pancake graphs. Together, these results settle two recent conjectures of Blanco and Buehrle.
\end{abstract}

    

\section{Introduction}

Let $[n]:=\{1,\dots,n\}$ and $\mathbb{Z}_m:=\mathbb{Z}/m\mathbb{Z}$. 
Denote by $S(m,n)$ the wreath product $\mathbb{Z}_m\wr S_n$, which one interprets as a group of \emph{coloured permutations}.
An element $\sigma\in S(m,n)$ will be written as a pair $\sigma=(\chi,\psi)$, where
$\chi:[n]\to\mathbb{Z}_m$ is the \emph{colour function} and $\psi\in S_n$ is the underlying permutation.
Equivalently, we identify $\sigma$ with its coloured one–line notation
\[
\sigma=\bigl(\psi(1)^{\chi(1)},\,\psi(2)^{\chi(2)},\,\ldots,\,\psi(n)^{\chi(n)}\bigr).
\]

Given $1\le i\le j\le n$ and $\varepsilon\in\{\pm 1\}$, the \emph{generalised substring reversal} is the map $s_{i,j}^{\varepsilon}:S(m,n)\to S(m,n)$ defined by $(\chi,\psi) \mapsto (\chi^\prime,\psi^\prime)$, where
\[
\psi'(t)=
\begin{cases}
\psi(i+j-t),& i\le t\le j,\\
\psi(t),& \text{otherwise},
\end{cases}\,\,\,
\chi'(t)=
\begin{cases}
\chi(i+j-t)+\varepsilon \ (\bmod m),& i\le t\le j,\\
\chi(t),& \text{otherwise}.
\end{cases}
\]
In other words, $s_{i,j}^{\varepsilon}$ reverses the sub-block $(\psi(i),\psi(i+1),\dots,\psi(j-1),\psi(j))$ of $\psi$ and adds $\varepsilon$ to $\chi(k)$ for each $k \in \{i,i+1,\dots,j-1,j\}$.
In the setting of \emph{generalised pancake flipping}, one only considers blocks that are prefixes: $r_k^{\varepsilon}:=s_{1,k}^{\varepsilon}$ for $1\le k\le n$. Note that, when $m=2$, we have $r_k^{+}=r_k^{-}$.

\begin{figure}[htbp]
	\centering
\begin{tikzpicture}[x=1cm,y=1cm, rotate=210, scale=0.80]

\tikzset{
cell/.style={draw=red!20, line width=1.15pt},
cell1/.style={draw=red!20, line width=1.15pt, fill=blue!10},
cell2/.style={draw=red!20, line width=1.15pt, fill=yellow!10},
r2/.style={draw=black!95,-},
r1/.style={draw=black!95,-},
Avertex/.style={circle,thin,draw=black!15,fill=black!90, inner sep=0pt, minimum width=6pt},
Bvertex/.style={circle,thin,draw=black!15,fill=black!90, inner sep=0pt, minimum width=6pt},
cluster/.style={draw=blue, line width=1.15pt}
}

\begin{scope}
    \coordinate (O0) at (0,6.2);
\coordinate (O1) at (5.36,3.1);
\coordinate (O2) at (5.36,-3.1);
\coordinate (O3) at (0,-6.2);
\coordinate (O4) at (-5.36,-3.1);
\coordinate (O5) at (-5.36,3.1);
\coordinate (I0) at (0,3.1);
\coordinate (I1) at (2.68,1.55);
\coordinate (I2) at (2.68,-1.55);
\coordinate (I3) at (0,-3.1);
\coordinate (I4) at (-2.68,-1.55);
\coordinate (I5) at (-2.68,1.55);

\draw[cell] (O0)--(O1)--(O2)--(O3)--(O4)--(O5)--cycle;
\draw[cell] (I0)--(I1)--(I2)--(I3)--(I4)--(I5)--cycle;
\draw[cell1] (O0) -- (I0) -- (I1) -- (O1) -- cycle;
\draw[cell2] (O1) -- (I1) -- (I2) -- (O2) -- cycle;
\draw[cell1] (O2) -- (I2) -- (I3) -- (O3) -- cycle;
\draw[cell2] (O3) -- (I3) -- (I4) -- (O4) -- cycle;
\draw[cell1] (O4) -- (I4) -- (I5) -- (O5) -- cycle;
\draw[cell2] (O5) -- (I5)  -- (I0) -- (O0) -- cycle;
\end{scope}

\begin{scope}[scale=0.95]
\coordinate (A00) at ( 1.80, 3.13);
\coordinate (A01) at ( 1.35, 4.89);
\coordinate (A02) at ( 3.55, 3.61);
\coordinate (A10) at (-3.61, 0.00);
\coordinate (A11) at (-4.92,-1.26);
\coordinate (A12) at (-4.92, 1.26);
\coordinate (A20) at ( 1.80,-3.13);
\coordinate (A21) at ( 3.55,-3.61);
\coordinate (A22) at ( 1.35,-4.89);

\coordinate (B00) at (-1.80,-3.13);
\coordinate (B01) at ( 3.61, 0.00);
\coordinate (B02) at (-1.80, 3.13);
\coordinate (B10) at (-1.35,-4.89);
\coordinate (B11) at ( 4.92, 1.26);
\coordinate (B12) at (-3.55, 3.61);
\coordinate (B20) at (-3.55,-3.61);
\coordinate (B21) at ( 4.92,-1.26);
\coordinate (B22) at (-1.35, 4.89);

\coordinate (LA00) at (1.26, 2.39);
\coordinate (LA01) at (1.90, 5.05);
\coordinate (LA02) at (3.42, 3.87);
\coordinate (LA10) at (-2.42, 0);
\coordinate (LA11) at (-5.23,-0.86);
\coordinate (LA12) at (-5.23, 0.86);
\coordinate (LA20) at (1.26, -2.39);
\coordinate (LA21) at ( 3.42,-3.87);
\coordinate (LA22) at ( 1.90,-5.05);

\coordinate (LB00) at (-1.77,-2.73);
\coordinate (LB01) at ( 3.17, 0);
\coordinate (LB02) at (-1.77, 2.73);
\coordinate (LB10) at (-1.91,-5.05);
\coordinate (LB11) at ( 5.35, 1.46);
\coordinate (LB12) at (-3.79, 3.89);
\coordinate (LB20) at (-3.79,-3.89);
\coordinate (LB21) at ( 5.35,-1.46);
\coordinate (LB22) at (-1.91, 5.05);

\node[Avertex] at (A00) {};
\node[Avertex] at (A01) {};
\node[Avertex] at (A02) {};
\node[Avertex] at (A10) {};
\node[Avertex] at (A11) {};
\node[Avertex] at (A12) {};
\node[Avertex] at (A20) {};
\node[Avertex] at (A21) {};
\node[Avertex] at (A22) {};

\node[Bvertex] at (B00) {};
\node[Bvertex] at (B01) {};
\node[Bvertex] at (B02) {};
\node[Bvertex] at (B10) {};
\node[Bvertex] at (B11) {};
\node[Bvertex] at (B12) {};
\node[Bvertex] at (B20) {};
\node[Bvertex] at (B21) {};
\node[Bvertex] at (B22) {};



\draw[r1] (A00)  -- (A10);
\draw[r1] (A00) -- (A20);
\draw[r1] (A01) edge [bend left = 18.01] (A11);
\draw[r1] (A01) edge [bend right = 18.08] (A21);
\draw[r1] (A02) edge [bend left = 18.01] (A12);
\draw[r1] (A02) edge [bend right = 18.08] (A22);
\draw[r1] (A10) -- (A20);
\draw[r1] (A11) edge [bend left = 18.01] (A21);
\draw[r1] (A12) edge [bend left = 18.01] (A22);

\draw[r1] (B00) -- (B10);
\draw[r1] (B00) -- (B20);
\draw[r1] (B01) -- (B11);
\draw[r1] (B01) -- (B21);
\draw[r1] (B02) -- (B12);
\draw[r1] (B02) -- (B22);
\draw[r1] (B10) -- (B20);
\draw[r1] (B11) -- (B21);
\draw[r1] (B12) -- (B22);

\draw[r2] (A00) -- (B11);
\draw[r2] (A00) -- (B22);
\draw[r2] (A01) -- (B02);
\draw[r2] (A01) -- (B21);
\draw[r2] (A02) -- (B01);
\draw[r2] (A02) -- (B12);
\draw[r2] (A10) -- (B12);
\draw[r2] (A10) -- (B20);
\draw[r2] (A11) -- (B00);
\draw[r2] (A11) -- (B22);
\draw[r2] (A12) -- (B02);
\draw[r2] (A12) -- (B10);
\draw[r2] (A20) -- (B10);
\draw[r2] (A20) -- (B21);
\draw[r2] (A21) -- (B01);
\draw[r2] (A21) -- (B20);
\draw[r2] (A22) -- (B00);
\draw[r2] (A22) -- (B11);

\node[Avertex] at (A00) {};
\node[Avertex] at (A01) {};
\node[Avertex] at (A02) {};
\node[Avertex] at (A10) {};
\node[Avertex] at (A11) {};
\node[Avertex] at (A12) {};
\node[Avertex] at (A20) {};
\node[Avertex] at (A21) {};
\node[Avertex] at (A22) {};

\node[Bvertex] at (B00) {};
\node[Bvertex] at (B01) {};
\node[Bvertex] at (B02) {};
\node[Bvertex] at (B10) {};
\node[Bvertex] at (B11) {};
\node[Bvertex] at (B12) {};
\node[Bvertex] at (B20) {};
\node[Bvertex] at (B21) {};
\node[Bvertex] at (B22) {};
\end{scope}

\begin{scope}
    \node[font=\small] at (4.3,3.1) {$V_{0,1}$};
\node[font=\small] at (-4.3,-3.1) {$V_{0,2}$};
\node[font=\small] at (-4.9,2.4) {$V_{1,1}$};
\node[font=\small] at (4.9,-2.4) {$V_{1,2}$};
\node[font=\small] at (0.4,-5.4) {$V_{2,1}$};
\node[font=\small] at (-0.4,5.4) {$V_{2,2}$};
\end{scope}

\end{tikzpicture}
\caption{Generalised pancake graph $\mathcal P_3(2)$ with its vertices grouped according to the vertex partition $\mathfrak P_{3,2}$.}
\label{fig:P32}
\end{figure}

The \textbf{generalised pancake graph} $\mathcal{P}_m(n)$ (see~\cite{BB3}) is the undirected Cayley graph $\mathcal{P}_m(n)=\mathrm{Cay}\left (S(m,n),R_m\right)$, where
\[
R_m=
\begin{cases}
\{\,r_k^{+}: 1\le k\le n\,\},& \text{ if } m=2, \\
\{\,r_k^{-} \;:\; 1\le k\le n\,\} \cup \{\,r_k^{+} \; : \; 1\le k\le n\,\},& \text{ if } m\ge3.
\end{cases}
\]
In other words, the vertices of the graph $\mathcal{P}_m(n)$ are all $m$–coloured permutations $S(m,n)$ where two vertices are adjacent if and only if one can be obtained from the other by applying a prefix reversal $r_k^{\varepsilon}$. In particular, $\mathcal{P}_1(n)$ is the classical \textbf{pancake graph}, while $\mathcal{P}_2(n)$ is the \textbf{burnt pancake graph} in which “adding $1\bmod2$’’ coincides with changing the sign of each flipped entry. For $m\geqslant 3$ the graph $\mathcal{P}_m(n)$ is $2n$–regular (two flips for each $k \in [n]$), whereas for $m=2$ it is $n$–regular (the two flips coincide).
In Figure~\ref{fig:P32}, the graph $\mathcal P_3(2)$ is illustrated.
The vertices are partitioned into vertex subsets
\[
V_{i,j}\;:=\;\bigl\{\,(\chi,\psi)\in S(3,2):\;
\chi(j)=i \text{ and } \psi(j)=1\,\bigr\}
\]
for each $i\in\{0,1,2\}$ and $j\in\{1,2\}$; these subsets are demarcated by trapezia. Within each trapezium, the vertex closest to the central hexagon corresponds to the coloured permutation in which $2$ has colour $0$. Proceeding clockwise from this vertex around the trapezium yields first the coloured permutation in which $2$ has colour $1$, followed by the coloured permutation in which $2$ has colour $2$.

Pancake graphs and their generalisations exhibit diverse properties, such as minimal flip counts~\cite{BFR}, pancyclicity~\cite{BB2} and applications in genome rearrangement~\cite{KS}, parallel computing~\cite{KF}, and information theory~\cite{HS}.
In genomics,  pancake flipping models chromosomal inversions, where extensive DNA segments reverse orientation, a process analogous to the burnt pancake sorting~\cite{HP, H}.

We address three spectral properties of the graphs $\mathcal P_m(n)$: the spectral gap, eigenvalue multiplicity, and integral eigenvalues. 
Our motivation comes from recent work~\cite{BB,BB3} of Blanco and Buehrle on the integral spectra of burnt pancake graphs and generalised pancake graphs.
In this note, we prove both~\cite[Conjecture 9(i)]{BB} (Theorem~\ref{thm:main}) and~\cite[Conjecture 6.1]{BB3} (Theorem~\ref{thm:mult}). 
We denote by $\lambda_i(\Gamma)$ the $i$th largest eigenvalue of the adjacency matrix of $\Gamma$.

\begin{theorem}[Blanco-Buehrle's spectral gap conjecture]\label{thm:main}
    Let $n, m \geqslant 2$ be integers.
    Then
    \[
    \lambda_1(\mathcal P_m(n)) - \lambda_2(\mathcal P_m(n)) < \begin{cases}
        1 & \text{ if $m = 2$;} \\
        2 & \text{ otherwise.}
    \end{cases}
    \]
\end{theorem}

Theorem~\ref{thm:main} also improves on \cite[Theorem 4.1]{BB3}, which implies an upper bound on the spectral gap: $\lambda_1(\mathcal P_m(n)) - \lambda_2(\mathcal P_m(n)) \le 2$.

For a graph $\Gamma$ with eigenvalue $\lambda$, denote by $\operatorname{mult}_{\Gamma}(\lambda)$ the multiplicity of $\lambda$ as an eigenvalue of the adjacency matrix of $\Gamma$.

\begin{theorem}[Blanco-Buehrle's multiplicity conjecture\footnote{We warn the reader of a typo in the statement of \cite[Conjecture 6.1]{BB3} where it should read $k \in [n-1] \backslash \{\lfloor n/2\rfloor \}$.}]\label{thm:mult}
Let $m\equiv 0\pmod{4}$ and $n\geqslant 2$.  
Then
\[
  \operatorname{mult}_{\mathcal P_m(n)}(2k)\;\geqslant\;
  \begin{cases}
    3, & 1\leqslant k\leqslant n-1,\;\, k\ne \left \lfloor \dfrac{n}{2} \right \rfloor;\\
    2, & k=\left \lfloor \dfrac{n}{2}\right \rfloor.
  \end{cases}
\]
\end{theorem}

The key to our proofs is a simple observation about the form of the quotient matrix of an equitable partition of $\mathcal P_m(n)$ (see Lemma~\ref{lem:circ(DCOOC)}). 
We note in passing that, even though the graphs $\mathcal P_m(n)$ are Cayley graphs, since the set of prefix reversals is not conjugation-closed, it is difficult to apply character theory~\cite{GJ, Z} directly to compute their eigenvalues.

\section{Some spectral properties of $\mathcal P_m(n)$}

\subsection{Equitable partitions and interlacing}

Given a real symmetric $n\times n$ matrix $M$, we denote its eigenvalues in nonincreasing order $\lambda_1(M) \geqslant \lambda_2(M) \geqslant \dots \geqslant \lambda_n(M)$.

\begin{lemma}[\cite{F}]\label{lem:cauchy}
Let $A$ be an $n \times n$ real symmetric matrix and $B$ be an $m\times m$ principal submatrix of $A$.
Then, for each $i=1,\dots,m$,
\[
  \lambda_i(A)\ \geqslant \ \lambda_i(B)\ \geqslant\ \lambda_{n-m+i}(A).
\]
\end{lemma}

Given a graph $\Gamma=(V,E)$, a partition $\mathfrak P=\{V_1,\dots,V_m\}$ of $V$ is called \textbf{equitable} if, for all $j=1, \dots,m$, the numbers $q_{j}(u):=|\Gamma(u)\cap V_j|$ do not depend on the choice of $u\in V_i$. 
These numbers, independent of the vertex, are denoted by $q_{ij}$ and form the $m\times m$ \textbf{quotient matrix} $Q(\mathfrak P) := (q_{ij})_{1\leqslant i,j\leqslant m}$. 
\begin{lemma}[{Godsil and Royle~\cite{GR}, Theorem 9.3.3}]
\label{lem:ep}
    Let $\Gamma$ be a graph with adjacency matrix $A$.
    Suppose that $\mathfrak P$ is an equitable partition of $\Gamma$.
    Then $\det(xI-Q(\mathfrak P))$ divides $\det(xI - A)$.
\end{lemma}

\subsection{The equitable partition $\mathfrak P_{m,n}$}

Let $m,n\in\mathbb{Z}$ with $m\geqslant 2$ and $n\geqslant 1$, and write $[n]:=\{1,2,\dots,n\}$.
Define
\[
D_n:=\operatorname{diag}(0,1,2,\dots,n-1)\in\mathbb{Z}^{n\times n},
\]
and let $E_n=(e_{ij})_{1\leqslant i,j\leqslant n}\in\{0,1\}^{n\times n}$ be given by
\[
e_{ij}\;=\;
\begin{cases}
1,&\text{if } i+j\leqslant n+1,\\
0,&\text{otherwise.}
\end{cases}
\]

For $i\in\{0,1,\dots,m-1\}$ and $j \in [n]$, define the vertex subsets
\[
V_{i,j}\;:=\;\bigl\{\,(\chi,\psi)\in S(m,n):\;
\chi(j)=i \text{ and } \psi(j)=1\,\bigr\}.
\]
Thus, $V_{i,j}$ consists of the coloured permutations in which the letter $1$ occupies position $j$ and has colour $i$ (equivalently, $\chi(\psi^{-1}(1))=i$).
Define the partition
\[
\mathfrak P_{m,n} := \{\,V_{i,j}:  i\in\{0,1,\dots,m-1\} ,\,\, j\in[n] \,\}
\]
of the vertex set of the (generalised) pancake graph.
See Figure~\ref{fig:P32} for an illustration of $\mathfrak P_{3,2}$.

Note that the equitable partition $\mathfrak P_{m,n}$ is the same as the one used by Blanco and Buehrle~\cite{BB3}.
However, by using the lexicographical ordering on the parts $V_{i,j}$ (see Lemma~\ref{lem:circ}), we can realise the quotient matrix as a circulant block matrix, which enables us to easily deduce more spectral properties of $\mathcal P_{m}(n)$.
In contrast, Blanco and Buehrle~\cite{BB,BB3} used a co-lexicographical ordering for their quotient matrices.

For $n\times n$ matrices $B_0,B_1,\dots,B_{m-1}$, we denote by $\operatorname{circ}\left(B_0,B_1,\dots,B_{m-1}\right)$ the circulant block matrix whose first row of blocks is $B_0,B_1,\dots,B_{m-1}$.
We write $O$ for the zero matrix with dimensions determined by context.

\begin{lemma}
\label{lem:circ}
Let $m \geqslant 2$ and $n \geqslant 1$ be integers. 
Then $\mathfrak P_{m,n}$ is an equitable partition of $\mathcal P_m(n)$ with quotient matrix 
\[
Q(\mathfrak P_{m,n})=\begin{cases} 
\operatorname{circ}\big( D_n,E_n \big) & m=2; \\
\operatorname{circ}\left(2D_n,E_n,\underbrace{O,\dots,O}_{m-3},E_n\right)  & m\geqslant 3.

\end{cases}
\]
\end{lemma}
\begin{proof}
Each element $\sigma=(\chi,\psi)\in S(m,n)$ corresponds to a unique ordered pair $(i,j)$ where $1$ is in position $j\in[n]$ with colour $i=\chi(j)\in\mathbb{Z}_m$ (equivalently, $i=\chi(\psi^{-1}(1))$). 
Hence, $\sigma$ lies exactly in one cell $V_{i,j}$; thus, the sets in $\mathfrak P_{m,n}$ are pairwise disjoint and cover the vertex set of $\mathcal P_m(n)$.

Let $x = (\chi_x,\psi_x) \in V_{i,j}$. 
Note that $1$ is in the $j$th position of $x$ with $\chi_x(j) = i$.
A prefix reversal $r_k^{\pm}$ fixes position $j$ of $x$ if and only if $k<j$. 
Such a flip reverses the prefix of length $k$ and adds $\pm1$ (mod $m$) to the colours inside that prefix, but leaves the entry $1^{i}$ at position $j$ unchanged. 
When $m\geqslant 3$, for each $k<j$, we have two types of reversals of the first $k$ positions: $r_k^{+}$ and $r_k^{-}$ (whereas $r_k^{+}=r_k^{-}$ when $m=2$). 
Therefore, the number of neighbours of $x$ that also belong to $V_{i,j}$ is
\[
d_j=\begin{cases}
2(j-1),& m\geqslant 3,\\
j-1,& m=2,
\end{cases}\qquad j=1,\dots,n.
\]
Moreover, whenever a prefix reversal $r_k^{\pm}$ moves the letter $1$ (i.e., $k\geqslant j$), the colour of $1$ changes from $i$ to $i\pm1\pmod m$, consequently, there are no edges from $V_{i,j}$ to any $V_{i,j^\prime}$ with $j^\prime\neq j$.
Note that the number of edges does not depend on the choice of $x \in V_{i,j}$.
Hence, we obtain diagonal blocks $2D_n$ of $Q(\mathfrak P_{m,n})$ for $m\geqslant 3$ (respectively $D_n$ for $m=2$). 

Now we consider the neighbours of $x$ that lie outside $V_{i,j}$.
Consider applying $r^{+}_k$ with $k\geqslant j$. 
After the flip, the position of $1$ changes from $j$ to $k-j+1$ and its colour changes from $i$ to $i+1\pmod m$.
Hence, $x$ is adjacent to (precisely one) vertex in $V_{i+1,k-j+1}$.
That is, $x$ has one neighbour in each of the parts $V_{i+1,j^\prime}$ for which $j + j^\prime \leqslant n+1$.
Note that, again, the number of neighbours does not depend on the choice of $x \in V_{i,j}$. 
Similarly, $x$ has one neighbour in each of the parts $V_{i-1,j^\prime}$ for which $j + j^\prime \leqslant n+1$.
Thus, we find that the $(i,i^\prime)$-block of $Q(\mathfrak P_{m,n})$ is equal to $E_n$, as required.
Hence, $\mathfrak P_{m,n}$ is an equitable partition of $\mathcal P_m(n)$ and $Q(\mathfrak P_{m,n})$ has the required form.
\end{proof}

The spectrum of a matrix $M$ is the multiset of its eigenvalues and is denoted by $\mathrm{Spec}\left(M\right)$.
The multiplicity-preserving union of two multisets is denoted by $\uplus$.
Next, we state a lemma that provides a formula for the spectrum of a circulant block matrix.



\begin{lemma}[{\cite[Theorem 6 (c)]{Friedman}}]
\label{lem:t6}
    Let $m,n\ge 2$ be integers and let $A_0,A_1,\dots,A_{m-1}\in \mathrm{Mat}_n(\mathbb C)$. 
Then 
\[
\mathrm{Spec}\left (\operatorname{circ}(A_0,A_1,\dots,A_{m-1})\right )=\biguplus_{j=0}^{m-1}\mathrm{Spec} \left (\sum_{i=0}^{m-1}\zeta^{ij}A_{i} \right ),
\] 
for some primitive $m$th root of unity $\zeta$.
\end{lemma}
    
The next lemma follows immediately from Lemma~\ref{lem:circ} via special cases of \cite[Theorem 10.3]{Bass94} and Lemma~\ref{lem:t6}. 

\begin{lemma}\label{lem:circ(DCOOC)}
Let $m,n \geqslant 2$ be integers.
Then
\begin{align}
\mathrm{Spec}\left(Q(\mathfrak P_{m,n})\right)=\begin{cases}
\mathrm{Spec}(D_n+E_n)\uplus \operatorname{Spec}(D_n-E_n) & m=2;\\
\biguplus_{k=0}^{m-1}\operatorname{Spec}\left(2D_n
+2\cos\tfrac{2\pi k}{m}E_n\right) & m\geqslant 3.
\end{cases}
\end{align}
        Moreover, if $m\equiv 0\pmod{4}$ then
        \[
        \operatorname{Spec}(2D_n)
  \uplus
  \operatorname{Spec}(2D_n)
  \uplus
  \operatorname{Spec}(2D_n+2E_n)\subseteq\mathrm{Spec}\left(Q(\mathfrak P_{m,n})\right).
        \]
\end{lemma}

\subsection{Proofs of the main results}

We are now ready to prove our main results.

\begin{proof}[Proof of Theorem~\ref{thm:main}]
First, suppose that $m\geqslant 3$.
It is straightforward to check that $\lambda_1(Q(\mathfrak P_{m,n})) = \lambda_1\left(2D_n
+2E_n\right)$.
Furthermore, by Lemma~\ref{lem:circ(DCOOC)}, we have
\begin{align*}
\lambda_2(Q(\mathfrak P_{m,n})) &\geqslant \lambda_1\left(2D_n
+tE_n\right)
\end{align*}
for some nonzero $t<2$.
Observe that $\left [ \begin{smallmatrix}
    t & t\\
t & 2n-2
\end{smallmatrix} \right ]$ is the $2\times2$ principal submatrix of $2D_n+tE_n$ induced on its first and last rows/columns.
Hence, by Lemma~\ref{lem:cauchy},
\[
\lambda_{1}(2D_n+tE_n) \geqslant \frac{t+2n-2+\sqrt{((2n-2)-t)^2+4t^2}}{2} 
>2n-2.
\]

Clearly, $\lambda_1(\mathcal {P}_m(n)) = \lambda_1(Q(\mathfrak P_{m,n})) =2n$, and, by Lemma~\ref{lem:ep}, we have $\lambda_2(\mathcal {P}_m(n)) \geqslant \lambda_2(Q(\mathfrak P_{m,n}))$.
Thus, we find that 
\[
 \lambda_1(\mathcal {P}_m(n))-\lambda_2(\mathcal {P}_m(n)) \leqslant  \lambda_1(\mathcal {P}_m(n))-\lambda_1(2D_n+tE_n) < \lambda_1(\mathcal {P}_m(n))-(2n-2)=2.
\]

The case of $m=2$ follows \textit{mutatis mutandis} using the principal submatrix $\left [\begin{smallmatrix}
-1 &-1\\
-1 & n-1
\end{smallmatrix}\right ]$ of $D_n-E_n$.
\end{proof}

\begin{proof}[Proof of Theorem~\ref{thm:mult}]
    By Lemma~\ref{lem:ep} and Lemma~\ref{lem:circ(DCOOC)}, we have
    \begin{align*}
        \operatorname{Spec}(\mathcal P_m(n)) \supseteq \operatorname{Spec}(Q(\mathfrak P_{m,n}))  \supseteq \operatorname{Spec}(2D_n) \uplus \operatorname{Spec}(2D_n) \uplus \operatorname{Spec}(2D_n + 2E_n). 
    \end{align*}
    Clearly, $\operatorname{Spec}(2D_n) = \{0,2,4,\dots,2(n-1)\}$.
    Using \cite[Proposition 4.1]{GSW}\footnote{In \cite[\S4]{GSW}, the Hermitian matrix $B_n$ denotes the adjacency operator of the Schreier initial-reversal graph, which is represented
by the matrix $D_n+E_n$ in our notation.}, we find that
     \[
     \operatorname{Spec}(2D_n+2E_n)=\{0,2,4,\dots,2n\}\setminus\left \{2\left \lfloor \frac{n}{2} \right \rfloor\right \}.
     \]
     This completes our proof.
\end{proof}

Observe that our proof of Theorem~\ref{thm:mult} also provides a simple proof of the main results of Blanco and Buehrle~\cite[Theorem 5]{BB} and \cite[Theorem 4.1]{BB3}.
Furthermore, using a similar idea, one can also provide an alternative, straightforward proof of \cite[Theorem 5.1]{BB3}.

Finally, we take the opportunity to combine and restate (part of) two separate conjectures of Blanco and Buehrle, \cite[Conjecture 9(ii)]{BB} and \cite[Conjecture 6.2]{BB3}\footnote{We warn the reader of a typo in the statement of \cite[Conjecture 6.2]{BB3} where it should read $m>2$.}.

\begin{conjecture}
Let $m\geqslant 2$ be an integer.
Then
\[
\lambda_1(\mathcal P_m(n)) - \lambda_2(\mathcal P_m(n)) \to \begin{cases}
    1 & \text{ if $m = 2$; } \\
    2 & \text{ if $m \geqslant 3$; }
\end{cases}
\text{ as } n \to \infty.
\]
\end{conjecture}

Moreover, we provide an additional conjecture from cursory computational evidence.

\begin{conjecture}
\label{con:2}
    Let $m \geqslant 1$ and  $n \geqslant 2$ be integers. 
    Then
    \[\lambda_1(\mathcal P_m(n)) - \lambda_2(\mathcal P_m(n)) =\lambda_1(Q(\mathfrak P_{m,n})) - \lambda_2(Q(\mathfrak P_{m,n})).\]
\end{conjecture}

Conjecture~\ref{con:2} has been proved by Cesi in~\cite{C} when $m=1$. 
For $m\geqslant 2$, Conjecture~\ref{con:2} is still open.

\bigskip
\noindent{\bf Acknowledgements.} The authors thank the referees for their careful reading and helpful suggestions, which have improved the paper.
GG was supported by the Singapore Ministry of Education Academic Research Fund; grant numbers: RG14/24 (Tier 1) and MOE-T2EP20222-0005 (Tier 2). 

\bibliographystyle{plain}

\end{document}